\documentclass[10pt,a4paper]{amsart}

\theoremstyle{plain}
\newtheorem{theorem}{Theorem}[section]
\newtheorem{lemma}[theorem]{Lemma}

\newtheorem{corollary}[theorem]{Corollary}
\theoremstyle{definition}

\theoremstyle{remark}
\newtheorem{remark}[theorem]{Remark}

\def\bin #1#2 {\left( \matrix { #1 \cr #2 \cr } \right) }

\begin{document}

\title[On the genus of  projective curves]
{On the genus of  projective curves not contained in hypersurfaces
of given degree, II}

\author{Vincenzo Di Gennaro }
\address{Universit\`a di Roma \lq\lq Tor Vergata\rq\rq, Dipartimento di Matematica,
Via della Ricerca Scientifica, 00133 Roma, Italy.}
\email{digennar@mat.uniroma2.it}

\author{Giambattista Marini }
\address{Universit\`a di Roma \lq\lq Tor Vergata\rq\rq, Dipartimento di Matematica,
Via della Ricerca Scientifica, 00133 Roma, Italy.}
\email{marini@mat.uniroma2.it}

\abstract Fix integers $r\geq 4$ and $i\geq 2$. Let $C$ be a
non-degenerate, reduced and irreducible complex projective curve in
$\mathbb P^r$, of degree $d$, not contained in a hypersurface of
degree $\leq i$. Let $p_a(C)$ be the arithmetic genus of $C$.
Continuing previous research, under the assumption $d\gg
\max\{r,i\}$, in the present paper we exhibit a Castelnuovo  bound
$G_0(r;d,i)$ for $p_a(C)$. In general, we do not know whether this
bound is sharp. However, we are able to prove it is sharp when
$i=2$, $r=6$ and $d\equiv 0,3,6$ (mod $9$). Moreover, when $i=2$,
$r\geq 9$, $r$ is divisible by $3$, and $d\equiv 0$ (mod
$r(r+3)/6$), we prove that if $G_0(r;d,i)$ is not sharp, then for
the maximal value of $p_a(C)$ there are only three possibilities.
The case in which $i=2$ and $r$ is not divisible by $3$ has already
been examined in the literature. We give some information on the
extremal curves.

Projection of a rational normal scroll surface. Veronese surface. Maximal rank.


\endabstract
\maketitle

\section*{Acknowledgments}
The authors acknowledge the MIUR Excellence Department Project awarded to the Department of Mathematics, University of Rome Tor
Vergata.

\section{Introduction}

Let $C$ be a non-degenerate, reduced and irreducible complex
projective curve in $\mathbb P^r$ ($r\geq 3$), of degree $d$. A
classical result of Castelnuovo's states that the arithmetic genus
$p_a(C)$ of $C$ satisfies the following bound:
\begin{equation}\label{GC}
p_a(C)\leq G(r; d):=\binom{m_1}{2}(r-1)+m_1\epsilon_1,
\end{equation}
where $m_1$ and $\epsilon_1$ are defined by dividing
$d-1=m_1(r-1)+\epsilon_1$, $0\leq \epsilon_1\leq r-2$. The bound
$G(r; d)$ is sharp. The extremal curves, also called {\it
Castelnuovo curves}, are contained in {\it quadric hypersurfaces}.
When $d>2r$, the intersection of all quadrics containing a given
Castelnuovo curve is a rational normal surface of degree $r-1$ in
$\mathbb P^r$. By inspection of curves on such surfaces, one
produces all Castelnuovo curves \cite[pp. 81-93]{EH}, \cite[p.
703]{CCD2}, \cite[p. 44]{H}. This theory can be generalized in
several ways. We refer to \cite{CCD2}, \cite{DGnnew}, and the
references therein, for an overview on this topic.

\smallskip
In the present paper we turn to the following {\lq\lq natural\rq\rq}
generalization: {\it fixed integers $r,d,i$, with $r\geq 4$, $i\geq
2$, $d\gg\max\{r,i\}$, find the maximal arithmetic genus
\begin{equation}\label{G}
G(r;d,i)
\end{equation}
for an integral curve $C$ of degree $d$ in $\mathbb P^r$, not
contained in a hypersurface of degree $\leq i$} \cite[p. 229, line
18-23 from above]{CCD}. This problem has remained completely open
until the recent papers \cite{DGnew}, \cite{DGnnew}. We think it
will be useful briefly summarize the results appearing in these
papers.

\smallskip
Divide:
\begin{equation}\label{f1}
\binom{r+i}{i}-(i+1)=\alpha\binom{i+1}{2}+\beta, \quad 0\leq
\beta\leq \binom{i+1}{2}-1.
\end{equation}
Set:
\begin{equation}\label{f2}
s_0=s_0(r,i):=\begin{cases} \alpha \quad{\text{if $\beta=0$}}\\
\alpha+1 \quad{\text{if $\beta>0$.}}
\end{cases}
\end{equation}
For every integer $d\geq 1$ define $m$ and $\epsilon$ by dividing
\begin{equation}\label{f4}
d-1=ms_0+\epsilon, \quad 0\leq \epsilon\leq s_0-1.
\end{equation}
The number $s_0$ is the {\it expected} minimal degree for an
integral surface $S\subset \mathbb P^r$ not contained in
hypersurfaces of degree $i$ (see Remark \ref{r1}, $(i)$, below).

\smallskip
The quoted papers \cite{DGnew}, \cite{DGnnew}  {\it deal only with
the case $\beta=0$} (for instance, when $i=2$, $\beta=0$ if and only
if $3$ does not divide $r$).

\smallskip
When $\beta=0$, sharp results are presented for $r=4$ and $i=2$, and
for $r\geq 5$ and $2\leq i\leq 3$. Moreover, when $r\geq 5$ and
$i\geq 2$, in \cite{DGnnew} one proves that, for a curve
$C\subset\mathbb P^r$ not contained in hypersurfaces of degree $\leq
i$, and of  degree $d\gg\max\{r,i\}$, one has:
\begin{equation}\label{f3}
p_a(C)\leq \binom{m}{2}s_0+m\epsilon, \end{equation}
where $m$ and $\epsilon$ are defined in (\ref{f4}).
By a numerical argument based on \cite[Main theorem]{CCD}, one reduces
the proof of the bound (\ref{f3}) to curves lying on surfaces
$S\subset \mathbb P^r$ of degree $s_0$. Next,  one applies \cite[Lemma
2.1]{DGnnew}, which states the following property. Even if it is
easy to prove, in our opinion it is a little surprising.

\smallskip
{\it If $S\subset \mathbb P^r$ ($r\geq 5$) is an integral surface of
degree $s_0$ (with $\beta=0$), not contained in hypersurfaces of
degree $\leq i$, then $h^0(S,\mathcal O_S(1))=s_0+2$, i.e. $S$ is an
isomorphic projection of a rational normal scroll $S'\subset
\mathbb P^{s_0+1}$} ($S'$ is a fortiori smooth).

\smallskip
Therefore, {\it if such a surface $S$ exists}, then the bound
(\ref{f3}) is sharp and the extremal curves are the Castelnuovo
curves coming from $S'$ by projection. Under this assumption (i.e.
the existence of $S$) a curious fact  follows, i.e. {\it the sharp
bound $G(d;r,i)$ previously defined in (\ref{G}) is equal to the
Castelnuovo bound $G(s_0+1;d)$ of the curves in $\mathbb P^{s_0+1}$
of degree $d$} (compare with (\ref{GC}) and (\ref{f3})):
\begin{equation}\label{curious}
G(r;d,i)=G(s_0+1;d).
\end{equation}
As for the question of the existence of the surface $S$, i.e. of the
effective sharpness of the bound (\ref{f3}), it reduces to the
following: {\it is it possible to project isomorphically a rational
normal scroll $S'\subset \mathbb P^{s_0+1}$ in a surface
$S\subset\mathbb P^r$ not contained in hypersurfaces of degree $\leq
i$?} By \cite[Theorem 2]{BE} the answer is {\it yes} when $2\leq
i\leq 3$. More precisely, by \cite[Theorem 2]{BE} one knows that a
general projection $S$ of $S'$
is of {\it maximal rank} under the assumption
$$
s_0>r\geq 5, \quad 6s_0+4\leq \binom{r+3}{3}.
$$
Since this condition is satisfied when $2\leq i\leq 3$, and
$h^0(S',\mathcal O_{S'}(i))=s_0\binom{i+1}{2}+(i+1)=\binom{r+i}{i}$,
it easily follows that $h^0(\mathbb P^r, \mathcal I_S(i))=0$ when
$2\leq i\leq 3$. We do not know whether the bound (\ref{f3}) is
sharp when $i\geq 4$. We stress the fact that, being $S'$ a rational
normal scroll, one knows the curves contained in it (see e.g.
\cite{EH}). In particular, one knows the existence on $S'$ of curves
of degree $d$ for {\it every} $\epsilon$ (compare with (\ref{f4})).

\smallskip
{\it In the present paper we  extend some of the previous results in
the case $\beta>0$ (compare with (\ref{f1})).}

\smallskip\noindent
When $i=2$, this is equivalent to say that $3$ divides $r$. This
case escapes to the analysis of \cite{DGnew}, \cite{DGnnew}
previously described. The case $\beta>0$ appears much more
complicated. In fact, in the range $\beta>0$, the argument which
leads to the proof of the quoted lemma \cite[Lemma 2.1]{DGnnew}
produces only two inequalities. If $S\subset \mathbb P^r$ is an
integral surface of degree $s_0$ not contained in hypersurfaces of
degree $\leq i$, then
\begin{equation}\label{in}
s_0-\pi+2-i(c_0-\pi)-\gamma \leq h^0(S,\mathcal O_S(1))\leq
s_0-\pi+2,
\end{equation}
where $\pi$ is the sectional genus of $S$, and $c_0$ and $\gamma$
are defined by dividing:
\begin{equation}\label{stiman}
\binom{i+1}{2}-\beta=c_0i+\gamma, \quad 0\leq \gamma\leq i-1
\end{equation}
(see Lemma \ref{stima} below; observe that $\pi\leq c_0$). Despite that, again by a numerical
argument (relying on the assumption $d\gg\max\{r,i\}$ and the main
results in \cite{CCD} and \cite{PAMS}), we are able to reduce the
problem of estimating $G(r;d,i)$ to the study of curves on a {\it
surface $S\subset\mathbb P^r$ of degree $s_0$, not contained in
hypersurfaces of degree $\leq i$, and with sectional genus
$\pi=c_0$} (see the proof of Theorem \ref{lastd}). In this case, the
inequalities (\ref{in}) say that $S$, if existing, is an isomorphic
projection of a surface $S'\subset \mathbb P^R$, with
$$s_0-c_0+1-\gamma \leq R \leq s_0-c_0+1.$$ It follows that the degree
$s_0$ of $S'$ in $\mathbb P^R$ is {\it low with respect to $R$} (see
Appendix below; when $r=4$ we need to assume $i\geq 4$ and $i\neq
15$):
$$
s_0\leq 2R-4.
$$
Therefore, we may apply  \cite[Theorem 2.2]{DF}, and  deduce the following bound
(compare with
(\ref{G}) and Remark \ref{r1}, $(iii)$):
\begin{equation}\label{G_0}
G(r;d,i)\leq  G_0(r;d,i),
\end{equation}
\begin{equation*}
G_0(r;d,i):=\binom{m}{2}s_0+m(\epsilon + c_0)+\binom{\gamma
+1}{2}+\max\left\{0,\,
\left\lfloor\frac{2c_0-(s_0-1-\epsilon)}{2}\right\rfloor\right\}-\mu,
\end{equation*}
where $\mu=0$ when $\gamma\leq 1$, and $\mu=1$ otherwise (Theorem
\ref{lastd} below). The numbers $s_0$, $m$, and $\epsilon$ are
defined in (\ref{f1}), (\ref{f2}) and (\ref{f4}), $c_0$ and $\gamma$
in (\ref{stiman}).

\smallskip
We are able to prove that the bound $G_0(r;d,i)$ is sharp for $i=2$,
$r=6$, and $\epsilon\in\{2,5,8\}$ (Theorem \ref{r=6}). When $i=2$,
$r\geq 9$, $r$ is divisible by $3$, and $\epsilon=s_0-1$, we prove
that if the bound $G_0(r;d,2)$ is not sharp, then for the sharp
value $G(r;d,2)$ there are only three possibilities (see Theorem
\ref{r9}, and Remark \ref{finale}, $(ii)$, for some comments on
other values of $\epsilon$). The proof of Theorem \ref{r=6} and
Theorem \ref{r9} relies on \cite[Theorem 3 and Theorem 2]{BE}, which
guarantees the existence of surfaces of degree $s_0$ with
$h^0(\mathbb P^r,\mathcal I_S(2))=0$. This implies that, for {\it
every} $0\leq \epsilon\leq s_0-1$, the extremal curves of genus
$G(r;d,2)$ are necessarily contained in surfaces $S\subset \mathbb
P^r$ of degree $s_0$ not contained in quadrics (Remark \ref{finale},
$(i)$).

\smallskip
When $i\geq 3$ (and $\beta>0$), the numerical assumptions in
\cite[loc. cit.]{BE} fail. Therefore, in the  case $i\geq 3$ we do
not know whether the bound $G_0(r;d,i)$ is sharp. However, when
$0\leq \gamma\leq 1$, we can say that if the bound $G_0(r;d,i)$ is
sharp, then an extremal curve lies   on a  surface $S\subset\mathbb
P^r$ of degree $s_0$, sectional genus $\pi=c_0$, arithmetic genus
$p_a(S)=-\gamma$, with $h^0(\mathbb P^r,\mathcal I_S(i))=0$.
Conversely, if there exists such a surface, at least when $\gamma=0$
and $\epsilon=s_0-1$, the bound $G_0(r;d,i)$ is sharp (Remark
\ref{r2}, $(i)$ and $(ii)$).

\smallskip
The question of the sharpness is divided into the following two
steps: 1) to establish whether there are surfaces $S$ of degree
$s_0$ in $\mathbb P^r$ with sectional genus $\pi=c_0$ and
$h^0(\mathbb P^r,\mathcal I_S(i))=0$, and which ones are they; 2) to
study the curves of maximal genus for the corresponding surface
$S'\subset\mathbb P^R$ of which $S$ is an isomorphic projection.
Recall that $S'$ is a surface of low degree in $\mathbb P^R$, in
fact $s_0\leq 2R-4$, with $s_0-c_0+1-\gamma \leq R \leq s_0-c_0+1$.
The simplest case seems to be the case $\gamma=0$ (for $\beta>0$, this happens
when $i=2$ and $3$ divides $r$, or when $i=3$ and $r=0,\,20,\,28$ (mod $36$)).
When $\gamma=0$, then a fortiori
$R=s_0-c_0+1$, and the surface $S'\subset \mathbb P^R$ is a {\it
Castelnuovo surface}, i.e. its sectional genus $\pi=c_0$ is equal to
the Castelnuovo's bound for curves of degree $s_0$ in $\mathbb
P^{R-1}$ (when $\pi=c_0=1$, such a surface is called {\it of almost
minimal degree} \cite{BS}). On such surfaces $S'\subset \mathbb P^R$
(at least when $\epsilon=s_0-1$) one is able to identify extremal
curves. In fact, these curves are {\it the extremal curves in
$\mathbb P^R$ with respect to the condition of being not contained
in surfaces of degree $<s_0$}, and are studied in \cite{EH} and
\cite{CCD}. If $\epsilon=s_0-1$, then they are exactly the complete
intersections of $S'$ with a hypersurface of degree $m+1$. In these
papers, in the general case $0\leq \epsilon\leq s_0-1$, the extremal
curves are exhibited mostly on cones, which cannot be projected in
an isomorphic way. This is another complication, because we do not
know whether on a {\it fixed} Castelnuovo surface there exist
extremal curves in the sense of \cite{EH} and \cite{CCD} for every
$0\leq \epsilon\leq s_0-1$.

Taking into account that if $\gamma=0$, then $h^0(S',\mathcal
O_{S'}(i))=\binom{r+i}{i}=h^0(\mathbb P^r,\mathcal O_{\mathbb
P^r}(i))$ (see (\ref{gbe}) below), previous discussion leads to the
question: {\it when $\gamma=0$, is it possible to project a
Castelnuovo surface $S'\subset \mathbb P^R$ isomorphically in a
surface $S\subset\mathbb P^r$ of maximal rank?} (see Remark
\ref{r2}, $(i)$ and $(ii)$). If this is possible, at least when
$\epsilon=s_0-1$, then $G_0(r;d,i)$ would be sharp, and equal to the
bound $\pi_{s_0-R+1}(d,R)$ appearing in \cite[p. 116]{EH}, which in
turn is equal to the bound $G(d,s_0,R)$ determined in \cite{CCD}.
This would be a generalization of (\ref{curious}), because a
rational normal scroll surface is a particular case of a Castelnuovo
surface.

\smallskip
Another interesting case would seem to be the case $c_0=0$ (e.g. $i=3$ and $r=5$ (mod $36$)).
In this case, the surfaces involved are those with sectional genus $c_0=0$.

\smallskip
In Remark \ref{r2}, $(iv)$, we make explicit the assumption
$d\gg\max\{r,i\}$. It is certainly not the best possible. One might
hope to do better with a closer careful and complicate examination
of the numerical functions arising in the proofs.  We decided not to
push here this investigation further.

\smallskip
We will keep the notations introduced in (\ref{G}), (\ref{f1}), (\ref{f2}),
(\ref{f4}), (\ref{in}), (\ref{stiman})  throughout the
paper. Observe that once the numbers $r,d,i$ are set,
the numbers $s_0$, $\beta$, $m$, $\epsilon$, $c_0$ and $\gamma$  are determined.

\smallskip
All the results presented, with the exception of the Lemma \ref{Clifford} and Corollary \ref{suff},
concern the case $\beta>0$.

\section{Preliminary results}

The following Lemma \ref{Clifford} is certainly well known. We prove
it for lack of a suitable reference. In Remark \ref{r1}, $(iii)$, we
recall the definition of the symbols  $\lceil \frac{x}{2}\rceil$ and
$\lfloor \frac{x}{2} \rfloor$ for an integer $x\geq 1$.

\smallskip
\begin{lemma}\label{Clifford} Let $\Sigma\subset \mathbb P^{r-1}$, $r\geq 2$, be an integral projective
curve of degree $s$ and arithmetic genus $\pi$. Then, for every
$j\geq 1$, one has:
$$
h^0(\Sigma,\mathcal O_{\Sigma}(j))\,\leq \,1+js-\min\left\{\pi,
\left\lceil \frac{js}{2}\right\rceil\right\}.
$$
\end{lemma}

\begin{proof}
First assume that $\pi\leq \left\lceil \frac{js}{2}\right\rceil$.

\medskip
In this case one has $2\pi-2-js<0$. It follows that $\deg
(\omega_{\Sigma}(-j))=2\pi-2-js<0$, where $\omega_{\Sigma}$ denotes
the dualizing sheaf \cite{Fujita}. Hence $h^1(\Sigma,\mathcal
O_{\Sigma}(j))=h^0(\Sigma,\omega_{\Sigma}(-j))=0$, and therefore
$1-\pi+js=h^0(\Sigma,\mathcal O_{\Sigma}(j))-h^1(\Sigma,\mathcal
O_{\Sigma}(j))=h^0(\Sigma,\mathcal O_{\Sigma}(j))$.

\medskip
Next assume $\pi > \left\lceil \frac{js}{2}\right\rceil$.

\medskip
If $h^1(\Sigma,\mathcal O_{\Sigma}(j))>0$, then, by Clifford's
Theorem for possibly singular curves \cite[Proposition 1.5]{Fujita},
we have $h^0(\Sigma,\mathcal O_{\Sigma}(j))\leq 1+\frac{js}{2}$.
Since $h^0(\Sigma,\mathcal O_{\Sigma}(j))$ is an integer,  it
follows that $h^0(\Sigma,\mathcal O_{\Sigma}(j))\leq 1+js-
\left\lceil \frac{js}{2}\right\rceil$. If $h^1(\Sigma,\mathcal
O_{\Sigma}(j))=0$, then $h^0(\Sigma,\mathcal O_{\Sigma}(j))=1-\pi+js
< 1+js- \left\lceil \frac{js}{2}\right\rceil$.
\end{proof}

\bigskip

\begin{corollary}\label{suff}
Let $S\subset \mathbb P^{r}$, $r\geq 2$, be a projective, integral
surface of degree $s$ and sectional genus $\pi$. Then, for every
$i\geq 1$, one has
\begin{equation}\label{i1}
h^0(\mathbb P^r,\mathcal I_S(i))\geq
\binom{r+i}{i}-\left[(i+1)+\binom{i+1}{2}s\right]+\sum_{j=1}^{i}c(j),
\end{equation}
where, for every $j\geq 1$, we set:
\begin{equation}\label{c(1)}
c(j):=\min\left\{\pi, \left\lceil \frac{js}{2}\right\rceil\right\}.
\end{equation}
\end{corollary}

\begin{proof}
By the exact sequence $0\to \mathcal I_S(i)\to \mathcal O_{\mathbb
P^r}(i)\to \mathcal O_S(i)\to 0$ we get
\begin{equation}\label{i2}
h^0(\mathbb P^r,\mathcal I_S(i))\geq \binom{r+i}{i}-h^0(S,\mathcal
O_{S}(i)). \end{equation} On the other hand, from the exact sequence
$0\to \mathcal O_S(j-1)\to \mathcal O_{S}(j)\to \mathcal
O_{\Sigma}(j)\to 0$ ($\Sigma$ denotes a general hyperplane section
of $S$), we get
\begin{equation}\label{i3}
h^0(S,\mathcal O_{S}(i))\leq 1+\sum_{j=1}^{i}h^0(\Sigma,\mathcal
O_{\Sigma}(j)).
\end{equation}
Replacing (\ref{i3}) into (\ref{i2}), by the upper bound for
$h^0(\Sigma,\mathcal O_{\Sigma}(j))$ given by Lemma \ref{Clifford}
we get (\ref{i1}).
\end{proof}

\bigskip

\begin{remark}\label{r1}
$(i)$ Keep the notations introduced in (\ref{f1}) and (\ref{f2}).
Notice that
$$
\binom{r+i}{i}-\left[(i+1)+\binom{i+1}{2}s_0\right]\leq 0.
$$
In fact, if $\beta=0$, at the left of the previous inequality we
have exactly $0$. If $\beta>0$, then
$$
\binom{r+i}{i}-\left[(i+1)+\binom{i+1}{2}s_0\right]=\beta-
\binom{i+1}{2}<0.
$$
Moreover, notice that for every integer $\sigma<s_0$ we have
$$
\binom{r+i}{i}-\left[(i+1)+\binom{i+1}{2}\sigma\right]>0.
$$
Therefore, by previous Corollary \ref{suff} it follows that {\it every
integral projective surface $S\subset \mathbb P^{r}$, $r\geq 4$, of
degree $\sigma<s_0$ is contained in a hypersurface of degree $i$}.

\smallskip
$(ii)$ Notice that for every $j\geq 1$ one has  $c(j)\geq c(1)$.
Therefore $\sum_{j=1}^{i}c(j)\geq ic(1)$.

\smallskip
$(iii)$  Let $\sigma\geq 1$ be an integer, and divide $\sigma=2h+k$,
$0\leq 1\leq k$. Recall that if $k=0$, then $\left\lceil
\frac{\sigma}{2}\right\rceil=h$, and that if $k=1$, then
$\left\lceil \frac{\sigma}{2}\right\rceil=h+1$. Moreover,
$\left\lfloor \frac{\sigma}{2}\right\rfloor=h$.

\smallskip
$(iv)$ With reference to Lemma \ref{Clifford}, taking into account
that $h^0(\Sigma,\mathcal O_{\Sigma}(j))\geq 1-\pi+js$, we see that
if $\pi\leq \frac{1}{2}(s+1)$, then $h^0(\Sigma,\mathcal
O_{\Sigma}(j))= 1-\pi+js$ for every $j\geq 1$.
\end{remark}

\medskip In the claim of next Lemma \ref{stima} we keep the notations
introduced in (\ref{f1}), (\ref{f2}) and (\ref{stiman}), where we
defined the numbers $s_0$, $\beta$, $c_0$ and $\gamma$.

\medskip
\begin{lemma}\label{stima}
Fix integers $r\geq 4$ and $i\geq 2$.  Let $S\subset\mathbb P^r$ be
an integral surface of degree $s_0=s_0(r,i)$, sectional genus $\pi$,
not contained in a hypersurface of $\mathbb P^r$ of degree $\leq i$.
Assume $\beta>0$. Then $2c_0<s_0$ and $0 \leq \pi\leq c_0$.
Moreover, one has
\begin{equation}\label{inf}
s_0-\pi+2-i(c_0-\pi)-\gamma\,\leq\, h^0(S,\mathcal O_S(1))\,\leq\,
s_0-\pi+2.
\end{equation}
\end{lemma}

\begin{proof} First we prove that $2c_0<s_0$. The proof consists of
an easy numerical computation which does not depend on the existence
of the surface $S$.

\smallskip
Observe that
$$
2c_0=\frac{2}{i}\left[\binom{i+1}{2}-\beta-\gamma\right] \leq
\frac{2}{i}\binom{i+1}{2} \leq i+1,
$$
and that
$$
 s_0=\frac{\binom{r+i}{i}-(i+1)-\beta}{\binom{i+1}{2}}+1.
$$
Therefore, in order to prove that $2c_0<s_0$ it suffices to prove
that
$$i< \frac{\binom{r+i}{i}-(i+1)-\beta}{\binom{i+1}{2}}.$$
Since $\beta\leq \binom{i+1}{2}$ and $r\geq 4$, previous inequality
follows from the inequality:
$$
i\binom{i+1}{2}+(i+1)+\binom{i+1}{2}<\binom{4+i}{i}.
$$
Dividing by $i+1$, we see that previous inequality is equivalent to:
$$
12i^2+12i+24<(4+i)(3+i)(2+i).
$$
And this is easy to prove because $(4+i)(3+i)(2+i)=i^3+9i^2+26i+24$.

\smallskip
Now we are going to prove that $\pi\leq c_0$.

Assume, by contradiction, that $\pi>c_0$. Since $2c_0<s_0$, we have
$c_0+1\leq \left\lceil \frac{s_0}{2}\right\rceil$. Therefore, we
have $c(1)=\min\left\{\pi, \left\lceil
\frac{s_0}{2}\right\rceil\right\}\geq c_0+1$ (compare with
(\ref{c(1)})). Taking into account Remark \ref{r1}, $(ii)$, by
Corollary \ref{suff} it follows that
$$
h^0(\mathbb P^r,\mathcal I_S(i))\geq
\binom{r+i}{i}-\left[(i+1)+\binom{i+1}{2}s_0\right]+i(c_0+1)=i-\gamma>0.
$$
This is in contrast with the assumption that $S$ is not contained in
a hypersurface of $\mathbb P^r$ of degree $\leq i$. This proves that
$\pi\leq c_0$.

\medskip
We conclude by proving the inequalities (\ref{inf}).

Let $\Sigma\subset\mathbb P^{r-1}$ be a general hyperplane section
of $S$. Since  $2c_0<s_0$ and $\pi\leq c_0$, we have $2\pi<s_0$. It
follows that $h^0(\Sigma, \mathcal O_{\Sigma}(j))= 1-\pi+js_0$ for
every $j\geq 1$ (compare with Remark \ref{r1}, $(iv)$). Combining
with the natural exact sequences $0\to \mathcal O_S(j-2)\to \mathcal
O_S(j-1)\to \mathcal O_{\Sigma}(j-1)\to 0$,  we get
$$
h^0(S,\mathcal O_S(1))\leq 2+s_0-\pi,\quad {\text{and}}
$$
\begin{equation}\label{start}
h^0(S,\mathcal O_S(i-1))\leq h^0(S,\mathcal O_S(1))+\sum_{j=2}^{i-1}
(1-\pi+js_0)\leq i+\binom{i}{2}s_0-\pi(i-1).
\end{equation}
Since $h^0(\mathbb P^r,\mathcal I_S(i-1))=0$, from the natural exact
sequence $0\to \mathcal I_S(i-1)\to \mathcal O_{\mathbb P^r}(i-1)\to
\mathcal O_{S}(i-1)\to 0$ we get:
$$
h^1(\mathbb P^r,\mathcal I_S(i-1))=h^0(S,\mathcal
O_S(i-1))-\binom{r+i-1}{i-1}
$$
$$
\leq i+\binom{i}{2}s_0-\pi(i-1)-\binom{r+i-1}{i-1}.
$$
Since $h^0(\mathbb P^r,\mathcal I_S(i))=0$, from the natural exact
sequence $0\to \mathcal I_S(i-1)\to \mathcal I_{S}(i)\to \mathcal
I_{\Sigma,\,\mathbb P^{r-1}}(i)\to 0$ we get:
$$
h^0(\mathbb P^{r-1},\mathcal I_{\Sigma,\,\mathbb P^{r-1}}(i))\leq
h^1(\mathbb P^{r},\mathcal I_{S}(i-1))\leq
i+\binom{i}{2}s_0-\pi(i-1)-\binom{r+i-1}{i-1}.
$$
From the natural exact sequence $0\to \mathcal I_{\Sigma,\,\mathbb
P^{r-1}}(i)\to \mathcal O_{\mathbb P^{r-1}}(i)\to \mathcal
O_{\Sigma}(i)\to 0$ we get
$$
h^1(\mathbb P^{r-1},\mathcal I_{\Sigma,\,\mathbb P^{r-1}}(i))=
h^0(\mathbb P^{r-1},\mathcal I_{\Sigma,\,\mathbb
P^{r-1}}(i))-\binom{r+i-1}{r-1}+h^0(\Sigma,\mathcal O_{\Sigma}(i))
$$
$$
\leq
i+\binom{i}{2}s_0-\pi(i-1)-\binom{r+i-1}{i-1}-\binom{r+i-1}{r-1}+(1-\pi+is_0)
$$
$$
=(i+1)+\binom{i+1}{2}s_0-\binom{r+i}{i}-i\pi=\binom{i+1}{2}-\beta-i\pi=i(c_0-\pi)+\gamma.
$$
Therefore, we may write
$$
h^0(\mathbb P^{r-1},\mathcal I_{\Sigma,\,\mathbb
P^{r-1}}(i))=\binom{r+i-1}{r-1}-h^0(\Sigma,\mathcal
O_{\Sigma}(i))+x, \quad 0\leq x \leq i(c_0-\pi)+\gamma.
$$
Hence, we have:
$$
h^1(\mathbb P^{r},\mathcal I_{S}(i-1))\geq
\binom{r+i-1}{r-1}-h^0(\Sigma,\mathcal O_{\Sigma}(i))+x,
$$
and so:
$$
h^0(S,\mathcal O_S(i-1))\geq
\binom{r+i-1}{i-1}+\binom{r+i-1}{r-1}-h^0(\Sigma,\mathcal
O_{\Sigma}(i))+x.
$$
From (\ref{start}) it follows that:
$$
h^0(S,\mathcal O_S(1))\geq
\binom{r+i-1}{i-1}+\binom{r+i-1}{r-1}-h^0(\Sigma,\mathcal
O_{\Sigma}(i))+x -\sum_{j=2}^{i-1} (1-\pi+js_0)
$$
$$
= \binom{r+i-1}{i-1}+\binom{r+i-1}{r-1}-(1-\pi+is_0)+x
-\sum_{j=2}^{i-1} (1-\pi+js_0)
$$
$$
=(s_0-\pi+2)-i(c_0-\pi)-\gamma+x.
$$
\end{proof}

\medskip
\begin{remark}\label{Gianni}
From the natural exact sequence $0\to \mathcal O_S\to \mathcal
O_S(1)\to \mathcal O_{\Sigma}(1)\to 0$ we have
$$
h^0(S,\mathcal O_S(1))=s_0-\pi+2-\dim\, {\text{Im}}\delta_1,
$$
where $\delta_1:H^0(\Sigma,\mathcal O_{\Sigma}(1))\to H^1(S,\mathcal
O_{S})$ is the natural map. A more accurate analysis of the proof of
previous Lemma \ref{stima} shows that
$$
\dim\, {\text{Im}}\delta_1=i(c_0-\pi)+\gamma-\left(h^1(\mathbb
P^{r-1}, \mathcal I_{\Sigma,\mathbb P^{r-1}}(i))+\dim\, {\text{Im}}
\tau_i+\sum_{j=2}^{i-1}\dim\, {\text{Im}} \delta_j\right),
$$
where
$$\delta_j:H^0(\Sigma,\mathcal O_{\Sigma}(j))\to
H^1(S,\mathcal O_{S}(j-1))\,\, {\text{and}}\,\, \tau_j:h^1(\mathbb
P^{r}, \mathcal I_{S}(j-1))\to h^1(\mathbb P^{r}, \mathcal I_{S}(j))
$$
are the natural maps.  Although we will not use previous formula, it
seems of some interest to us.
\end{remark}

\section{Genus of curves not contained in hypersurfaces of given degree}

As about Theorem \ref{lastd} below,  the integer $G_0(r;d,i)$ is defined in (\ref{G_0}),
for an explicit inequality concerning the assumption $d\gg \max\{r,i\}$ see Remark \ref{r2},
$(iv)$. Observe that from  Theorem \ref{lastd} it follows
that $G(r;d,i)\leq G_0(r;d,i)$ (compare with (\ref{G})).

\begin{theorem} \label{lastd} Fix integers $r\geq 4$ and $i\geq 2$, with $\beta >0$
(when $r=4$ assume $i\geq 4$, $i\neq 15$). Let $C\subset \mathbb
P^r$ be a projective, non-degenerate, integral  curve of degree
$d\gg \max\{r,i\}$, not contained in hypersurfaces of degree $\leq
i$. Let $p_a(C)$ be the arithmetic genus of $C$. Then $p_a(C)\leq
G_0(r;d,i)$.
\end{theorem}

\begin{proof} Let $S\subset \mathbb P^r$ be an integral surface
containing $C$, of the minimal degree $s$ with respect to such
condition. By Remark \ref{r1}, $(i)$, we know that $s\geq s_0$.

\smallskip
If $s>s_0$, then $C$ is not contained in surfaces of degree
$<s_0+1$. By \cite[Main theorem and first line below the claim]{CCD}
we know that
$$
p_a(C)\leq \frac{d^2}{2(s_0+1)}+O(d).
$$
On the other hand, an elementary computation proves that
$$
G_0(r;,d,i)=\frac{d^2}{2s_0}+O(d).
$$
Therefore, for $d\gg \max\{r,i\}$, we have
\begin{equation}\label{(a)}
p_a(C)\leq
\frac{d^2}{2(s_0+1)}+O(d)<\frac{d^2}{2s_0}+O(d)=G_0(r;d,i).
\end{equation}
It follows that, in order to prove our Theorem \ref{lastd}, we may
assume that $s=s_0$, i.e. we may assume that $C$ is contained in a
surface  $S\subset \mathbb P^r$ of degree exactly $s_0$.

\smallskip
In this case, by Lemma \ref{stima} we know that $\pi\leq c_0$. By
\cite[Lemma and inequality (2.1)]{PAMS} we have
$$
p_a(C) \leq \frac{d^2}{2s_0}+\frac{d}{2s_0}(2\pi-2-s_0)+O(1).
$$
On the other hand, taking into account (\ref{f4}), an elementary
computation proves the following identity:
$$
\frac{d^2}{2s_0}+\frac{d}{2s_0}(2\pi-2-s_0)+\frac{1+\epsilon}{2s_0}(s_0+1-\epsilon-2\pi)=
\binom{m}{2}s_0+m(\epsilon + \pi).
$$
We deduce:
$$
p_a(C) \leq \binom{m}{2}s_0+m(\epsilon + \pi)+O(1).
$$
Therefore, taking into account (\ref{f4}) and that $d\gg
\max\{r,i\}$, if $\pi<c_0$ we have:
\begin{equation}\label{(b)}
p_a(C) < \binom{m}{2}s_0+m(\epsilon + c_0)+O(1)=G_0(r;d,i).
\end{equation}
Hence we may assume $\pi=c_0$.

\smallskip
Summing up, in order to prove our Theorem \ref{lastd}, we may assume
that $C$ is contained in a surface  $S\subset \mathbb P^r$ of degree
$s_0$ and sectional genus $\pi=c_0$.

\smallskip
In this case, by Lemma \ref{stima} we know that $S$ is an isomorphic
projection of a surface $S'\subset \mathbb P^R$ of degree $s_0$,
sectional genus $\pi=c_0$, with $h^0(S',\mathcal O_{S'}(1))=R+1$,
and $$s_0-c_0+1-\gamma\leq R\leq s_0-c_0+1.$$ A direct computation
proves that (see Appendix at the end of the paper; when $r=4$,  here
we need the hypothesis  $i\geq 4$ and $i\neq 15$):
$$
s_0\leq 2R-4.
$$
Denote by $C'\subset S'\subset \mathbb P^R$ the curve corresponding
to $C\subset S\subset \mathbb P^r$. Since $d\gg \max\{r,i\}$, by
Bezout's theorem the surface $S'$ is the unique surface of degree
$\leq s_0$ containing $C'$. Therefore,  the curve $C'$ falls within
the hypotheses of \cite[Theorem 2.2]{DF}. We deduce (compare also
with \cite[Notations 2.1, $(ii)$]{DF})
$$
p_a(C)=p_a(C')\leq
G^*\left(R,d,s_0,c_0,-\binom{s_0-R-c_0+2}{2}\right)
$$
$$
:=\binom{m}{2}s_0+m(\epsilon + c_0)+\binom{s_0-R-c_0+2}{2}
+\max\left\{0,\,
\left\lfloor\frac{2c_0-(s_0-1-\epsilon)}{2}\right\rfloor\right\}.
$$
Since $s_0-c_0+1-\gamma\leq R$, i.e. $$s_0-R-c_0+2\leq \gamma +1,$$
we get (compare with (\ref{G_0}))
$$
p_a(C)=p_a(C')\leq G^*\left(R,d,s_0,c_0,-\binom{\gamma+1}{2}\right)
=G_0(r;d,i)+\mu.
$$

\smallskip
We conclude the proof observing that previous bound $p_a(C)\leq
G_0(r;d,i)+\mu$ cannot be sharp if $\mu=1$ (i.e. if $\gamma>1$).

\smallskip
In fact, if $p_a(C')=G_0(r;d,i)+\mu$, previous argument shows  that
$R$ is necessarily equal to $s_0-c_0-\gamma+1$:
$$
R=s_0-c_0-\gamma+1.
$$
Moreover, $C'\subset S'\subset \mathbb P^R$ should be an extremal
curve in the sense of \cite[Theorem 2.2]{DF}. Therefore, we also
have
$$p_a(S')=-\binom{\gamma+1}{2}=-\binom{s_0-R-c_0+2}{2}.$$ In this case,
we know that the Hartshorne-Rao Module of $S'$ vanishes: $M(S')=0$
\cite[Remark 3.7]{DF}. Since $M(S')=0$ we have
$$
h^0(S,\mathcal O_{S}(i))=h^0(S',\mathcal O_{S'}(i))=h_{S'}(i),
$$
where $h_{S'}$ denotes the Hilbert function of $S'\subseteq \mathbb
P^R$. Furthermore, we have (compare with \cite[Lemma 3.4 and Remark
3.5]{DF})
$$
h_{S'}(i)=\sum_{j=0}^{i}\,\, h_{S'}(j)-h_{S'}(j-1)=\sum_{j=0}^{i}
h_{\Sigma'}(j).
$$
On the other hand, by \cite[Proposition 3.1 and Remark 3.7]{DF} we
have, for every $j\geq 1$, $h^1(\Sigma',\mathcal O_{\Sigma'}(j))=0$,
and
$$
h_{\Sigma'}(j)=h^0(\Sigma',\mathcal O_{\Sigma'}(j))-h^1(\mathbb
P^{R-1}, \mathcal I_{\Sigma',\,\mathbb P^{R-1}}(j))
$$
$$
=(1-c_0+js_0)-\max\{0,s_0-R-c_0+2-j\}=(1-c_0+js_0)-\max\{0,\gamma+1-j\}.
$$
Summing up, we get (compare with (\ref{f1}), (\ref{f2}), and
(\ref{stiman}))
\begin{equation}\label{gbe}
h^0(S,\mathcal O_{S}(i))=h^0(S',\mathcal
O_{S'}(i))=\binom{r+i}{i}+\gamma -\binom{\gamma+1}{2}.
\end{equation}
Therefore, if $\gamma>1$,  then $h^0(\mathbb P^{r}, \mathcal
I_{S}(i))>0$. This is impossible.
\end{proof}

\medskip
\begin{remark}\label{r2} $(i)$
From the proof of Theorem \ref{lastd} we can deduce the following.
Assume $0\leq \gamma\leq 1$. In this case
$G_0(r;d,i)=G^*\left(R,d,s_0,c_0,-\binom{\gamma+1}{2}\right)$, with
$\gamma+1=s_0-R-c_0+2$ \cite[Theorem 2.2]{DF}. If the bound
$G_0(r;d,i)$ is sharp, then there exists an integral surface
$S\subset \mathbb P^r$ of degree $s_0$, sectional genus $\pi=c_0$,
and arithmetic genus $p_a(S)=-\binom{\gamma+1}{2}$, not contained in
hypersurfaces of degree $\leq i$. Every extremal curve $D$ is not
a.C.M., and it is contained in a flag $S\subset T\subset \mathbb
P^r$, where $S$ is a surface as before, uniquely determined by $D$,
not contained in hypersurfaces of degree $\leq i$, and $T$ is
hypersurface of $\mathbb P^r$ of degree $i+1$ by (\ref{i1}).
Furthermore, $S$ is the isomorphic projection in $\mathbb P^r$ of an
integral surface $S'$ of degree $s_0$ in $\mathbb
P^{s_0-c_0+1-\gamma}$ and, via this isomorphism, $D$ corresponds to
a curve $D'$ of $S'$ of degree $d$, whose arithmetic genus $p_a(D')$
is maximal with respect to the condition of being not contained in
surfaces of $\mathbb P^{s_0-c_0+1-\gamma}$ of degree strictly less
than $s_0$, and of degree $s_0$ with sectional genus $>c_0$
\cite[Theorem 2.2]{DF}. When $\gamma=0$ the surface $S'\subset
\mathbb P^{s_0-c_0+1}$ is a {\it Castelnuovo surface}, i.e. its
sectional genus is equal to the Castelnuovo's bound for curves
$\Sigma'$ of degree $s_0$ in $\mathbb P^{s_0-c_0}$. From (\ref{gbe})
we get
$$
h^0(S',\mathcal O_{S'}(i))=\binom{r+i}{i}.
$$
Therefore, the following question naturally arises: {\it when
$\gamma=0$, is it possible to project a Castelnuovo surface
$S'\subset \mathbb P^R$ isomorphically in a surface $S\subset\mathbb
P^r$ of maximal rank?}. This could lead to a generalization of
\cite[Theorem 2]{BE}. We also notice that extremal curves on
Castelnuovo surfaces are studied in \cite[p. 102]{EH}, \cite[pp.
242-244]{CCD}, and \cite[Section 4]{DF}. Unfortunately, these
surfaces are mostly cones on Castelnuovo curves. Hence, they cannot
be projected isomorphically.

\smallskip
$(ii)$ Conversely, in the case $\gamma=0$, assume there exists an
integral surface $S\subset \mathbb P^r$ of degree $s_0$, sectional
genus $\pi=c_0$, and arithmetic genus $p_a(S)=0$
($=-\binom{\gamma+1}{2}$), not contained in hypersurfaces of degree
$\leq i$. Then, at least when $\epsilon=s_0-1$, the bound
$G_0(r;d,i)$ is sharp. In fact, for every curve $C$ complete
intersection of $S$ with a hypersurface of degree $m+1$, we have
(compare with \cite[Remark 4.2, $(v)$]{DF})
$$
p_a(C)=\binom{m}{2}s_0+m(\epsilon + c_0)+ c_0=G_0(r;d,i).
$$
Therefore, combining with previous remark $(i)$, it follows that,
when $d\gg \max\{r,i\}$, $\gamma=0$ and $\epsilon=s_0-1$, {\it the
bound $G_0(r;d,i)$ is sharp if and only if there exists an integral
surface $S\subset \mathbb P^r$ of degree $s_0$, sectional genus
$\pi=c_0$, and arithmetic genus $p_a(S)=0$, not contained in
hypersurfaces of degree $\leq i$}. This property should be compared
with \cite[Theorem 1.1]{DGnnew}.

\smallskip
$(iii)$ Let $C$ be a curve as in Theorem \ref{lastd}. Since $C$ is
not contained in a surface of degree $<s_0$, we may apply the bound
$G(d,s_0,r)$ appearing in \cite[Main Theorem]{CCD}:
$$
p_a(C)\leq G(d,s_0,r)=\binom{m}{2}s_0+m(\epsilon+\pi^*)+O(1),
$$
where $\pi^*:=G(r-1;s_0)$ is the Castelnuovo's bound for curves of
degree $s_0$ in $\mathbb P^{r-1}$. Since
$$\pi^*=\frac{s_0^2}{2(r-2)}+O(s_0),$$ this estimate $G(d,s_0,r)$ is coarse
compared with
$$
G_0(r;d,i)=\binom{m}{2}s_0+m(\epsilon+c_0)+O(1).
$$
In fact,  by Lemma \ref{stima} we know that $2c_0<s_0$. Compare with
\cite[Remark 2.9]{CCD2}.

\smallskip
$(iv)$ During the proof of Theorem \ref{lastd}, we need to assume
$d\gg \max\{r,i\}$ in order to apply \cite[Main theorem]{CCD},
\cite[Lemma and inequality (2.1)]{PAMS}, and \cite[Theorem 2.2]{DF}.
Moreover, we need the assumption $d\gg \max\{r,i\}$ in order to establish
the inequalities (\ref{(a)}) and (\ref{(b)}). Taking this into account,
an elementary computation, that we omit,
proves that it suffices to assume:
$$
d>d_0(r,i):=\max\left\{
\frac{2(s_0+1)}{r-2}\prod_{j=1}^{r-2}[(r-1)!(s_0+1)]^{\frac{1}{r-1-j}},
\, 2^{s_0+4},\, 12(s_0+2)^4\right\}.
$$
Recall that $s_0$ is a function of $r$ and $i$, compare with
(\ref{f1}) and (\ref{f2}). This assumption $d>d_0(r,i)$ is also
sufficient later, in the proof of Theorem \ref{r=6} and of Theorem
\ref{r9}.
\end{remark}

\section{Genus of curves not contained in quadrics}

In this section we turn our attention to the case $i=2$, i.e. to
the genus of curves not contained in quadrics. We refer to
\cite{DGnew} and \cite{DGnnew} for the case in which $r$ is not
divisible by $3$. Therefore, we only examine the case in which $3$
divides $r$. In this case we have $\gamma=0$. We will prove that, for $r=6$ and
$\epsilon \in\{2,5,8\}$, the bound $G_0(6;d,2)$ appearing in
Theorem \ref{lastd} is sharp (see Theorem \ref{r=6} below and compare with (\ref{G_0})). When $r\geq
9$ ed $\epsilon=s_0-1$, we  prove that for the maximum value
$G(r;d,2)$ there are only four possibilities (see Theorem \ref{r9}
below, and compare with (\ref{G})). As we said in the Introduction, this "uncertainty" is due to the fact that we don't know
whether certain surfaces, of "small" degree, can be projected
isomorphically  onto surfaces not contained in quadrics, and what are the
curves on such surfaces. However,  we are able to  prove that, for {\it every} $0\leq \epsilon\leq
s_0-1$, the extremal curves are necessarily contained
in surfaces $S\subset \mathbb P^r$ of degree $s_0$, not contained in
quadrics (Remark \ref{finale}, $(i)$).

\smallskip
If $3$ divides $r$ we have:
$$
\binom{r+2}{2}=3s_0+1, \quad \beta=1, \quad c_0=1, \quad \gamma=0.
$$
Moreover, by Lemma \ref{stima} we know that if $S\subset \mathbb
P^r$ is a surface of degree $s_0$ not contained in quadrics, then
its sectional genus $\pi$ satisfies the condition
$$
0\leq \pi\leq 1.
$$
We also know that:
$$
\pi=1\implies h^0(S, \mathcal O_S(1))=s_0+1,
$$
and
$$
\pi=0\implies s_0\leq h^0(S, \mathcal O_S(1))\leq s_0+2.
$$
This implies that a surface of degree $s_0$ of $\mathbb P^r$ not
contained in quadrics must come, via isomorphic projection, from a
surface $S'\subset \mathbb P^{R}$, $R\in\{s_0-1,s_0,s_0+1\}$, of
degree $s_0$, with $h^1(\mathbb P^R, \mathcal I_{S'}(1))=0$.  When $s_0=R$, the surfaces
$S'$ are also called {\it of almost minimal degree} \cite{BS}. They
are relatively few, and have been classified in some way: \cite[p.
102]{EH}, \cite[p. 37]{Z}, \cite[Theorem 1.2, p. 349]{BS}.

\smallskip
We start with the case $r=6$. Next we analyze the case $r\geq 9$. We
could make a single statement, but for clarity we prefer to
distinguish the two cases. As for the assumption $d\gg r$, we refer to Remark \ref{r2}, $(iv)$.

\medskip
\begin{theorem}\label{r=6}
Assume $r=6$, $i=2$, $d\gg 0$, and $\epsilon \in\{2,5,8\}$. Then the
bound appearing in Theorem \ref{lastd} is sharp, i.e. $G(6;
d,2)=G_0(6;d,2)$.
\end{theorem}

\begin{proof} When $r=6$ we have
$s_0=9$, $\beta=1$, $c_0=1$, $\gamma=0$, $d-1=9m+\epsilon$, $0\leq \epsilon\leq 8$.
Therefore, we have
$$
G_0(6;d,2)=9\binom{m}{2}+m(\epsilon+1)+\nu,
$$
where $\nu=1$ if $\epsilon=8$, and $\nu=0$ otherwise.

In view of Theorem \ref{lastd}, we only have to prove the bound is sharp. To this purpose, let
$S'\subset\mathbb P^9$ the Veronese surface corresponding to the
linear system of plane cubic curves. By \cite[Theorem 3]{BE} we know
that a general projection  $S$ of $S'$ in $\mathbb P^6$ has maximal
rank. Since
$$
h^0(\mathbb P^6,\mathcal O_{\mathbb P^6}(2))=h^0(\mathbb
P^2,\mathcal O_{\mathbb P^2}(6))=h^0(S',\mathcal O_{S'}(2))=
h^0(S,\mathcal O_{S}(2)),
$$
it follows that the restriction map
$$
H^0(\mathbb P^6,\mathcal O_{\mathbb P^6}(2))\to H^0(S,\mathcal
O_{S}(2))
$$
is injective. Therefore, $S$ is not contained in quadrics. By
Bezout's theorem, every curve of $S$ of degree $d>18$ is not
contained in quadrics. Now the proof follows from the following
facts.

\smallskip
1) The curves $D$ on $S'$ have degree $d$ which is divisible by $3$,
correspond to plane curves $E$ of degree $\frac{d}{3}$, and
$p_a(D)=p_a(E)=\binom{\frac{d}{3}-1}{2}$;

\smallskip
2) the curves $D$ on $S'$ project isomorphically in curves of $S$ of
degree $d$;

\smallskip
3) $d$ is divisible by $3$ if and only if $\epsilon =2,5,8$;

\smallskip
4) when $\epsilon =2,5,8$, an elementary  direct computation proves
that
$$
\binom{\frac{d}{3}-1}{2}=\frac{d^2}{18}-\frac{d}{2}+1
=9\binom{m}{2}+m(\epsilon+1)+\nu=G_0(6;d,2).
$$
\end{proof}

Recall that the numbers $G(r;d,2)$ and $G_0(r;d,2)$ appearing in the claim
of next Theorem \ref{r9} are defined in (\ref{G}) and in (\ref{G_0}).

\medskip
\begin{theorem}\label{r9}
Let $r\geq 9$ be an integer divisible by $3$. Assume $i=2$, $d\gg
r$, and $\epsilon= s_0-1$. The integer $G(r;d,2)$ must assume one of the following four
values:
$$
G_0(r;d,2)-m-1, \quad G_0(r;d,2)-m, \quad G_0(r;d,2)-m+1, \quad G_0(r;d,2).
$$
\end{theorem}

\begin{proof}
In this case we have
$$
\binom{r+2}{2}=3s_0+1, \quad \beta=1, \quad c_0=1, \quad \gamma=0,
$$
and
$$
G_0(r;d,2)=\binom{m}{2}s_0+m(\epsilon+1)+1.
$$
Let $S'\subset \mathbb P^{s_0+1}$ be a smooth rational normal scroll
surface of degree $s_0$. Since
$$
3s_0+1=\binom{r+2}{2} \implies 6s_0+4\leq \binom{r+3}{3},
$$
by \cite[Theorem 2]{BE} we know that a general projection $S$ of
$S'$ in $\mathbb P^r$ has maximal rank. Since
$$
h^0(\mathbb P^r,\mathcal O_{\mathbb
P^r}(2))=\binom{r+2}{2}=3s_0+1<3s_0+3=h^0(S,\mathcal O_{S}(2)),
$$
the restriction map
$$
H^0(\mathbb P^r,\mathcal O_{\mathbb P^r}(2))\to H^0(S,\mathcal
O_{S}(2))
$$
is injective. It follows that $S$ is not contained in quadrics. By
Bezout's theorem this holds true also for the curves in $S$ of
degree $d\gg 0$. Projecting Castelnuovo curves of $S'$ we get curves
$C$ in $S$ with arithmetic genus
$$
p_a(C)= \binom{m}{2}s_0+m\epsilon=G_0(r;d,2)-m-1.
$$
It follows that $G(r;d,2)\geq G_0(r;d,2)-m-1$. Moreover, since
$G_0(r;d,2)-m-1=\frac{d^2}{2s_0}+O(d)$, the same argument we used in the
proof of Theorem \ref{lastd} shows that the extremal curves of genus
$G(r;d,2)$ are necessarily contained in surfaces $S\subset\mathbb
P^r$  of degree $s_0$ not contained in quadrics. And this holds true
for every $0\leq \epsilon\leq s_0-1$.

\smallskip
Let $S\subset\mathbb P^r$ be a surface of degree $s_0$ not contained
in quadrics. We know that its sectional genus $\pi$ is $0$ or $1$.
If $\pi=1$, taking a complete intersection of $S$ with a
hypersurface of degree $m+1$, we obtain curves with arithmetic genus
$G_0(r;d,2)$. Therefore,  by Theorem \ref{lastd}, in this case
$G(r;d,2)=G_0(r;d,2)$.

\smallskip
If such a surface does not exist, then the extremal curves are
contained in some surface $S\subset\mathbb P^r$ of degree $s_0$ not
contained in quadrics with sectional genus $\pi=0$, and $S$ is an
isomorphic projection of a surface $S'\subset \mathbb P^{R}$ of
degree $s_0$ with $s_0-1\leq R\leq s_0+1$. By
\cite[Theorem 2.2 and Notations 2.1, $(ii)$]{DF} we know that
the arithmetic genus of curves of $S'$ is bounded by $G_0(r;d,2)-m+2$
(when applying \cite[Theorem 2.2]{DF} observe that $s_0-R-\pi+2\leq 3$ because
$R\geq s_0-1$, and equality holds if and only if $R=s_0-1$).
The value $G_0(r;d,2)-m+2$ cannot be attained, otherwise
by \cite[Theorem 2.2, $(c)$]{DF} we would have $p_a(S)=p_a(S')=-3$. This is impossible, because
it implies $h^0(S, \mathcal O_S(2))<h^0(\mathbb P^r, \mathcal O_{\mathbb P^r}(2))$ (see
Lemma \ref{L3} below).
\end{proof}

\begin{lemma} \label{L3}
Let $S\subset \mathbb P^r$ be an integral, non-degenerate surface of
degree $s\geq 9$ with $3s+1=\binom{r+2}{2}$, and sectional genus
$\pi=0$. Assume that $h^0(S,\mathcal O_S(1))=s$. Then $3s\leq
h^0(S,\mathcal O_S(2))\leq 3s+1$. Moreover, if $h^0(S,\mathcal
O_S(2))=3s+1$ (which happens if $S$ is not contained in quadrics)
then  $p_a(S)=-2$. If $h^0(S,\mathcal O_S(2))=3s$ then $p_a(S)=-3$.
\end{lemma}

\begin{proof}
If we denote by $\Sigma$ a general hyperplane section of $S$
(recall that $p_a(\Sigma)=0$), from the exact sequence
$
0\to \mathcal O_S(-1)\to \mathcal O_{S}\to \mathcal O_{\Sigma}\to 0
$
we get:
$$
h^0(S,\mathcal O_S(2))\leq (1+2s)+s=3s+1.
$$
On the other hand, denote by $S'\subset \mathbb P^{s-1}$ the image
of $S$ via the complete linear system $|\mathcal O_S(1)|$, and by
$\Sigma'\subset \mathbb P^{s-2}$ its general hyperplane section. By
\cite[Proposition 3.1]{DF} we know that ($\mathcal
I_{\Sigma'}=\mathcal I_{\Sigma', \mathbb P^{s-2}}$)
$$
h^1(\mathbb P^{s-2},\mathcal I_{\Sigma'}(1))=2, \quad h^1(\mathbb P^{s-2},\mathcal
I_{\Sigma'}(2))\in\{0,1\},\quad h^1(\mathbb P^{s-2},\mathcal
I_{\Sigma'}(j))=0\,\,\,\, \forall\,\, j\geq 3.
$$
So, we have (we denote by $h_{S'}$ and $h_{\Sigma'}$ the Hilbert function of $S'$ and of $\Sigma'$)
$$
h^0(S,\mathcal O_S(2))=h^0(S',\mathcal O_{S'}(2))\geq h_{S'}(2)
$$
$$
\geq h_{S'}(1)+h_{\Sigma'}(2)=s+(1+2s)- h^1(\mathbb P^{s-2},\mathcal
I_{\Sigma'}(2))\geq 3s.
$$
We notice also that from the exact sequence
$
0\to \mathcal I_{S'}(-1)\to \mathcal I_{S'}\to \mathcal
I_{\Sigma'}\to 0
$
we deduce the exact sequence
$$
0=H^1(\mathbb P^{s-1},\mathcal I_{S'}(1))\to H^1(\mathbb P^{s-1},\mathcal I_{S'}(2))\to H^1(\mathbb P^{s-2},\mathcal
I_{\Sigma'}(2)),
$$
and that
$$
H^0(\mathbb P^{s-1},\mathcal
I_{S'}(2))=H^0(\mathbb P^{s-2},\mathcal I_{\Sigma'}(2)).
$$
In particular
$$
h^1(\mathbb P^{s-1},\mathcal I_{S'}(2))\in\{0,1\}.
$$

\smallskip
Now assume $h^0(S,\mathcal O_S(2))=3s+1$. Observe that this happens
if $S$ is not contained in quadrics (we do not know whether the
converse holds true).

\bigskip
If $h^1(\mathbb P^{s-1},\mathcal I_{S'}(2))=0$, then, comparing the exact sequences
$$
0\to \mathcal I_{S'}\to \mathcal O_{\mathbb P^{s-1}}\to \mathcal
O_{S'}\to 0,\quad 0\to \mathcal I_{\Sigma'}\to \mathcal O_{\mathbb
P^{s-2}}\to \mathcal O_{\Sigma'}\to 0,
$$
we get
$$
h^0(\mathbb P^{s-1},\mathcal I_{S'}(2))=\binom{s+1}{2}-(3s+1)=h^0(\mathbb P^{s-2},\mathcal
I_{\Sigma'}(2))
$$
$$
=\binom{s}{2}-(1+2s)+h^1(\mathbb P^{s-1},\mathcal I_{\Sigma'}(2)).
$$
We deduce
$$
h^1(\mathbb P^{s-1},\mathcal I_{\Sigma'}(2))=0.
$$
Therefore, the Hartshorne-Rao module $M(\Sigma')$ of $\Sigma'$ is equal to
$H^1(\mathbb P^{s-2},\mathcal I_{\Sigma'}(1))$ and has dimension  $\dim M(\Sigma')=2$.
On the other hand, since $h^1(\mathbb P^{s-2},\mathcal
I_{\Sigma'}(j))=0$ for every $j\geq 3$,
from the exact sequence
$0\to \mathcal I_{S'}(-1)\to \mathcal I_{S'}\to \mathcal
I_{\Sigma'}\to 0$ it follows that every map $H^1(\mathbb P^{s-1},\mathcal I_{S'}(j-1))\to
H^1(\mathbb P^{s-1},\mathcal I_{S'}(j))$ is onto for $j\geq 3$. Therefore, $M(S')=0$, and from \cite[Lemma 3.4]{DF}
we get
$$
p_a(S)=p_a(S')=-\dim M(\Sigma')+\sum_{j=1}^{+\infty}\mu_j=-2,
$$
where
$$\mu_j:=\dim\left[\ker(H^1(\mathbb P^{s-1},\mathcal I_{S'}(j-1))\to
H^1(\mathbb P^{s-1},\mathcal I_{S'}(j)))\right].
$$

\bigskip
If $h^1(\mathbb P^{s-1},\mathcal I_{S'}(2))=1$, then $h^1(\mathbb P^{s-2},\mathcal
I_{\Sigma'}(2))=1$. And so $\dim M(\Sigma')=3$ and a fortiori
$\sum_{j=1}^{+\infty}\mu_j=1$. As before, it follows that $p_a(S)=-2$.

\bigskip
Finally assume $h^0(S,\mathcal O_S(2))=3s$. From the sequence $0\to
\mathcal I_{S'}\to \mathcal O_{\mathbb P^{s-1}}\to \mathcal
O_{S'}\to 0$ we get
$$
h^0(\mathbb P^{s-1},\mathcal I_{S'}(2))=\binom{s+1}{2}-3s+h^1(\mathbb P^{s-1},\mathcal I_{S'}(2)).
$$
On the other hand we have
$$
h^0(\mathbb P^{s-2},\mathcal I_{\Sigma'}(2))=\binom{s}{2}-(1+2s)+h^1(\mathbb P^{s-2},\mathcal
I_{\Sigma'}(2)).
$$
Since $h^0(\mathbb P^{s-1},\mathcal I_{S'}(2))=h^0(\mathbb P^{s-2},\mathcal I_{\Sigma'}(2))$, we get
$$
h^1(\mathbb P^{s-2},\mathcal I_{\Sigma'}(2))-h^1(\mathbb P^{s-1},\mathcal I_{S'}(2))=1.
$$
A fortiori
$$
h^1(\mathbb P^{s-2},\mathcal I_{\Sigma'}(2))=1 \quad {\text{and}}\quad h^1(\mathbb P^{s-1},\mathcal
I_{S'}(2))=0.
$$
Therefore $\dim M(\Sigma')=3$, $M(S')=0$, and $p_a(S')=-3$.
\end{proof}

\begin{remark}\label{finale}
$(i)$ By the proof of Theorem \ref{r9} we deduce that the extremal
curves of genus $G(r;d,2)$ are necessarily contained in surfaces
$S\subset\mathbb P^r$  of degree $s_0$ not contained in quadrics.
And this holds true for {\it every} $0\leq \epsilon\leq s_0-1$.
Moreover, for $0\leq \epsilon < s_0-1$ we also have
$$
G_0(r;d,2)-m\leq G(r; d,2)\leq
G_0(r;d,2).
$$
The same
argument applies when $r=6$. In particular, when $r=6$ and
$\epsilon \notin\{2,5,8\}$, then $G_0(6;d,2)-m\leq G(6; d,2)\leq
G_0(6;d,2)$.
Since $m=(d-1-\epsilon)/s_0$, we deduce, for every $\epsilon$ and $r\geq 6$ divisible by $3$,
$$
G(r; d,2)=\frac{3d^2}{r(r+3)}+O(d).
$$
By \cite{DGnew} and \cite{DGnnew} we know that if $r\geq 4$ is not divisible by $3$, then
$$
G(r; d,2)=\frac{3d^2}{(r-1)(r+4)}+O(d).
$$
We see that the asymptotic behavior of the maximal genus $G(r;d,2)$
depends on the class of $r$ modulo $3$. Compared with classical
Castelnuovo-Halphen theory (\cite{EH}, \cite{CCD}), this is a
novelty.

\smallskip
$(ii)$ Taking into account \cite[p. 102]{EH}, \cite[p. 37]{Z},
\cite[Theorem 1.2, p. 349]{BS}, by an accurate analysis of the
surfaces with sectional genus $\pi=1$, it might be possible to prove
Theorem \ref{r9}  for other values of $\epsilon$, for instance
$\epsilon=0$ and $\epsilon \geq s-3$. We hope to return on this
question in a future paper.

\smallskip
$(iii)$ With reference to Lemma \ref{L3}, we proved that if $S$ is
not contained in quadrics, then $p_a(S)=-2$. We do not know whether
the converse is true.
\end{remark}

\section{Appendix}

With the assumptions as in Theorem \ref{lastd}, we are going to
prove that if $s_0-c_0+1-\gamma\leq R\leq s_0-c_0+1$, then
\begin{equation}\label{diseg}
s_0\leq 2R-4.
\end{equation}
We may assume $R=s_0-c_0+1-\gamma$. Therefore, previous inequality
is equivalent to
\begin{equation}\label{ssmall}
s_0-2(c_0+\gamma+1)\geq 0.
\end{equation}
Taking into account the definition of $s_0$ and $c_0$ (see (\ref{f1}), (\ref{f2}) and (\ref{stiman})), by
substituting one sees that inequality (\ref{ssmall}) is equivalent to
\begin{equation}\label{generale}
\binom{r+i}{i}-\frac{(i+1)(i^2+2i+2)}{2}+i\beta-\gamma(i+1)(i-1)\geq
0.
\end{equation}

\medskip
We are going to distinguish two cases: $r\geq 5$ and $r=4$.

\medskip
{\bf The case $r\geq 5$}.

When $r\geq 5$, taking into account that  $\beta>0$ and that $0\leq
\gamma\leq i-1$, in order to prove (\ref{generale}) it suffices to
prove that
$$
\binom{5+i}{i}-\frac{(i+1)(i^2+2i+2)}{2}+i-(i+1)(i-1)^2\geq 0.
$$
Omitting $i$ and dividing by $i+1$, it suffices to prove that
$$
\frac{1}{120}(5+i)(4+i)(3+i)(2+i)- \frac{i^2+2i+2}{2}-(i-1)^2\geq 0.
$$
Hence, in order to prove (\ref{diseg}) it suffices to prove that
$$
p(i):= i^4+14i^3-109i^2+274i-120\geq 0
$$
for every $i\geq 2$. This is obvious for $i\geq 8$. For the remaining values we have:
$$
p(2)=120, \quad p(3)=180, \quad p(4)=384
$$
$$
p(5)=900, \quad p(6)=1920, \quad p(7)=3660.
$$
This concludes the proof of (\ref{diseg}) in the case $r\geq 5$.

\medskip
{\bf The case $r= 4$}.

Starting from (\ref{generale}), since $\beta>0$, in order to prove (\ref{diseg}) it suffices to
prove that
$$
\binom{4+i}{i}-\frac{(i+1)(i^2+2i+2)}{2}+i-(i+1)(i-1)^2\geq
0
$$
for every $i\geq 4$, $i\neq 15$ . This is equivalent to check that
$$
q(i):=i^4-26i^3+23i^2+50i-24\geq 0.
$$
This is obvious for $i\geq 26$.
For the values $3\leq i\leq 25$ we may compute directly
$s_0-2(c_0+\gamma+1)$ (compare with (\ref{ssmall})). Since $\beta>0$, we only have to examine
$$
i\in\left\{3, 4,6,7,8,9,11,12,15,16,18,19,20,21,23,24\right\}.
$$
We get:

if $i=3$, then $s_0-2(c_0+\gamma+1)=6-2(1+2+1)=-2$;

if $i=4$, then $s_0-2(c_0+\gamma+1)=7-2(1+1+1)=1$;

if $i=6$, then $s_0-2(c_0+\gamma+1)=10-2(1+1+1)=4$;

if $i=7$, then $s_0-2(c_0+\gamma+1)=12-2(2+0+1)=6$;

if $i=8$, then $s_0-2(c_0+\gamma+1)=14-2(2+2+1)=4$;

if $i=9$, then $s_0-2(c_0+\gamma+1)=16-2(1+6+1)=0$;

if $i=11$, then $s_0-2(c_0+\gamma+1)=21-2(3+0+1)=13$;

if $i=12$, then $s_0-2(c_0+\gamma+1)=24-2(5+5+1)=2$;

if $i=15$, then $s_0-2(c_0+\gamma+1)=33-2(6+10+1)=-1$;

if $i=16$, then $s_0-2(c_0+\gamma+1)=36-2(4+4+1)=18$;

if $i=18$, then $s_0-2(c_0+\gamma+1)=43-2(3+3+1)=29$;

if $i=19$, then $s_0-2(c_0+\gamma+1)=47-2(5+0+1)=35$;

if $i=20$, then $s_0-2(c_0+\gamma+1)=51-2(5+5+1)=29$;

if $i=21$, then $s_0-2(c_0+\gamma+1)=55-2(3+14+1)=19$;

if $i=23$, then $s_0-2(c_0+\gamma+1)=64-2(6+0+1)=50$;

if $i=24$, then $s_0-2(c_0+\gamma+1)=69-2(10+10+1)=27$.

\bigskip
In conclusion, we proved that:
$$
r\geq 5 \quad{\text{and}}\quad  i\geq 2 \implies s\leq 2R-4
$$
and
$$
r=4, \quad  i\geq 4 \quad {\text{and}}\quad  i\neq 15 \implies s\leq
2R-4.
$$

\bigskip
{\bf{Declarations}}

\smallskip
{\bf{Conflict of interest}} The authors Vincenzo Di Gennaro and Giambattista Marini declare that there are no  interests
that are directly or indirectly related to the work submitted for
publication.

The authors Vincenzo Di Gennaro and Giambattista Marini did not receive support from any
organization for the submitted work.

The authors Vincenzo Di Gennaro and Giambattista Marini have no relevant financial or
non-financial interests to disclose.

\end{document}